\documentclass [11pt,oneside]{amsart}

\usepackage{comment}
\usepackage{amssymb} \usepackage{amsfonts} \usepackage{amsmath}
\usepackage{amsthm} \usepackage{epsfig} 
\usepackage{amsmath,amssymb,enumerate,pb-diagram,pb-xy}%,pdfsync}

\usepackage[applemac]{inputenc}
\input xy
\xyoption{all}

\usepackage{caption}

\newtheorem{lemma}{Lemma}[section]
\newtheorem{thm}[lemma]{Theorem}
\newtheorem{prop}[lemma]{Proposition}
\newtheorem{cor}[lemma]{Corollary} 
 
\newtheorem{prop_intro}{Proposition}
\newtheorem{cor_intro}[prop_intro]{Corollary}
\newtheorem{thm_intro}[prop_intro]{Theorem}

\newtheorem*{thm*}{Theorem}

\theoremstyle{definition}
\newtheorem{defn}[lemma]{Definition}

\newtheorem{example}[lemma]{Example}
\newtheorem{rem}[lemma]{Remark}

\newcommand{\matN}{\ensuremath {\mathbb{N}}}
\newcommand{\R} {\ensuremath {\mathbb{R}}}
\newcommand{\Q} {\ensuremath {\mathbb{Q}}}
\newcommand{\Z} {\ensuremath {\mathbb{Z}}}

\newcommand{\matH} {\ensuremath {\mathbb{H}}}

\newcommand{\vare} {\ensuremath{\varepsilon}}
\newcommand{\vol} {{\rm Vol}}
\newcommand{\vola} {{\rm Vol}_{\rm alg}}

\newcommand{\wdtM}{\widetilde{M}}

\newcommand{\calM} {\ensuremath {\mathcal{M}}}

\newcommand{\mathno}{\ensuremath{\overline{\matH^n}}}

\newcommand{\str} {\ensuremath {{\rm str}}}
\newcommand{\alt} {\ensuremath {{\rm alt}}}
\newcommand{\strtil} {\ensuremath {\widetilde{\rm str}}}

\newcommand {\bb} {\partial}

\author{Michelle Bucher}
\author{Roberto Frigerio}
\author{Cristina Pagliantini}

\address{Section de Math\'ematiques, 2-4 rue du Lièvre, Case postale 64, 1211 Genève 4, Suisse}
\email{michelle.bucher-karlsson@unige.ch}

\address{Dipartimento di Matematica \\
Universit\`a di Pisa \\
Largo B.~Pontecorvo 5 \\
56127 Pisa, Italy}

\email{frigerio@dm.unipi.it}

\address{Department Mathematik, ETH Zentrum, 8092 Zürich}
\email{cristina.pagliantini@math.ethz.ch}

\thanks{Michelle Bucher was supported by Swiss National Science Foundation 
project PP00P2-128309/1.  The authors thank the Institute Mittag-Leffler in Djursholm, Sweden, 
for their warm hospitality during  the preparation of this paper.}

\title[The simplicial volume of $3$-manifolds with boundary]{The simplicial volume of $3$-manifolds\\ with boundary}

\subjclass[2000]{}


\thanks{}

\begin{document}

\begin{abstract}
We provide sharp lower bounds for the simplicial volume of compact $3$-manifolds in terms 
of the simplicial volume of their boundaries. As an application, we 
compute the simplicial volume of several classes of $3$-manifolds, 
including handlebodies and products of surfaces with the interval.
Our results provide the first exact computation 
of the simplicial volume of a compact manifold 
whose boundary has positive simplicial volume. We also compute the 
minimal number of tetrahedra in a (loose) triangulation of the product of a surface with the interval.
%exact value of the 
% $\Delta$-complexity of products of surfaces with the interval. 
\end{abstract}
\maketitle

\section*{Introduction}

The simplicial volume  
 is an invariant of manifolds introduced by Gromov in his seminal paper~\cite{Gromov}. If $M$ is a connected, compact, oriented manifold with
(possibly empty) boundary, then the simplicial volume of $M$ is the 
infimum of the sum of the absolute values of the coefficients over all singular chains representing the real fundamental cycle of $M$ (see Section~\ref{preliminary}). 
It is usually denoted by $\|M\|$ if $M$ is closed, and by $\|M,\bb M\|$
if $\bb M\neq\emptyset$.
If $M$ is open, the fundamental class and the simplicial 
volume of $M$ admit analogous definitions in the context of homology of \emph{locally finite} chains, but in this paper we will restrict our attention  to compact manifolds:
unless otherwise stated, henceforth every manifold is assumed to be compact. Observe that the simplicial volume of an oriented manifold does not depend on its orientation and that it naturally extends also to  nonorientable or disconnected manifolds.

%As Gromov pointed out~\cite{Gromov}, 
Even if it depends only on the homotopy type of a manifold, the simplicial volume is deeply related
to the geometric structures that a manifold can carry. For example, closed manifolds which support 
negatively curved Riemannian metrics have nonvanishing simplicial volume, while the simplicial 
volume of flat or spherical manifolds is null (see \emph{e.g.}~\cite{Gromov}). 
Several vanishing and nonvanishing results
for the simplicial volume
are available by now,
but the exact value of nonvanishing simplicial volumes is known only in a very few cases. 
If $M$ is (the natural compactification of) a complete finite-volume hyperbolic $n$-manifold
without boundary, then 
a celebrated result by Gromov and 
Thurston implies that the simplicial volume of $M$ is equal to the Riemannian volume of $M$
divided by the volume $v_n$ of the regular ideal geodesic $n$-simplex in hyperbolic space (see \cite{Gromov,Thurston} for the compact case and \emph{e.g.}~\cite{stefano,FP,FM,BB} for the cusped case).
The only other exact computation of nonvanishing simplicial volume is for the product of two closed hyperbolic surfaces
or more generally manifolds locally isometric to the product of two hyperbolic planes \cite{Bucher3}. Building on these examples, more values for the simplicial volume can be obtained by surgery or by taking 
connected sums or amalgamated sums over submanifolds with amenable fundamental group (see \emph{e.g.}~\cite{Gromov, Kuessner,  BBFIPP}) however not by taking products. 

For hyperbolic manifolds
with geodesic boundary, it is proved by Jungreis~\cite{Jung} that if $M$ is such a manifold
and $\bb M\neq \emptyset$, then
$\|M,\bb M\|$ strictly exceeds $\vol(M)/v_n$,
and the last two authors showed that there exist, in any dimension, examples 
for which $\vol(M)/\|M,\bb M\|$
is arbitrarily close to $v_n$~\cite{FP}. 
These results were the sharpest estimates so far
for the simplicial volume of
manifolds whose boundary has
positive simplicial volume. We provide here the first exact computations of $\|M,\bb M\|$
for classes of $3$-manifolds for which $\| \bb M\|>0$.

%\subsection*{The simplicial volume of closed $3$-manifolds}
%In the $3$-dimensional case, the simplicial volume 
%of any closed $3$-dimensional manifold is equal
%to the sum of the simplicial volumes of the hyperbolic pieces
%of the canonical decomposition of the manifold.
%More precisely, let us say that a manifold is \emph{hyperbolic} if its internal part admits
%a complete finite-volume hyperbolic structure.
%If $M$ is a closed orientable $3$-manifold, then
%$M$ uniquely decomposes as the connected sum of a finite number
%of  prime manifolds. Moreover, the JSJ Decomposition Theorem and Perelman's proof 
%of Thurston's Geometrization Conjecture imply
%that each of these pieces
%contains a unique (up to isotopy) minimal collection of disjointly embedded incompressible tori such that each component of the $3$-manifold obtained by cutting along the tori is either hyperbolic or Seifert fibered. 
%Summing up, every closed $3$-manifold can be canonically decomposed into
%the union of hyperbolic and Seifert fibered pieces. It is known
%that the simplicial volume is additive with respect to gluings along spheres or incompressible tori. Since the simplicial volume of Seifert fibered spaces vanishes this implies 
%that the simplicial volume 
%of any closed $3$-manifold is equal
%to the sum of the simplicial volumes of its hyperbolic pieces, which are, for  toric boundary components, proportional to the Riemannian hyperbolic volumes of the pieces.

\subsection*{The simplicial volume of $3$-manifolds with boundary}
%Much less is known about the simplicial volume of $3$-manifolds with
%boundary components of negative Euler characteristic.
If $M$ is a connected oriented $n$-manifold with boundary, then the usual boundary map
takes any relative fundamental cycle of $M$
to the sum of fundamental cycles of the components of $\partial M$. 
As a consequence, for any $n$-manifold $M$ 
we have
\begin{equation}\label{boundary1:eq}
\|  M,\bb M\|\geq \frac{\| \partial  M\|}{n+1}\ .
\end{equation}
In particular, if $\|\partial M\|>0$, then $\| M,\bb M\|>0$. 

\begin{comment}
We improve this bound in Proposition~\ref{boundary2:prop} by replacing the factor $n+1$ by $n-1$ when $n\geq 2$, 
after observing that good cycles for the simplicial volume do  not have more than $n-1$ faces in the boundary. 
\end{comment}

The main result of this paper concerns $3$-dimensional manifolds, and provides
a sharp
lower bound for $\|M, \bb M\|$ in terms of $\|\partial M\|$:

\begin{thm_intro}\label{main:thm}
 Let $M$ be a $3$-manifold. Then there is a sharp inequality
$$
\| M,\bb M\| \geq \frac{3}{4} \| \partial M\|\ .
$$
\end{thm_intro}
The fact that the bound of Theorem~\ref{main:thm} is sharp is an immediate consequence
of Theorem~\ref{general:seif}. Theorems~\ref{main:thm} and~\ref{general:seif} are proved in Section~\ref{main:sec}. We will see in Theorem \ref{aspherical3:thm} that in the case of a boundary irreducible
 aspherical $3$-manifold, the constant $3/4$ can be improved to $5/4$. 

\subsection*{Stable $\Delta$-complexity and simplicial volume}
If $M$ is an $n$-manifold, we denote by
$\sigma(M)$ the \emph{$\Delta$-complexity} of $M$,
\emph{i.e.}~the minimal number of top dimensional simplices in a triangulation of $M$. We employ here the word ``triangulation'' in a loose sense, as is customary in geometric topology:
 a triangulation is the realization of $M$ as the gluing of finitely many $n$-simplices via some simplicial pairing of their codimension-$1$ faces. It is easy to see that the inequality
$\| M,\bb M\|\leq \sigma (M)$ holds (see \emph{e.g.}~\cite[Proposition 0.1]{FFM} or the discussion in the proof of Theorem~\ref{prod:surf:thm}
in Section~\ref{aspherical:sec}). 

The simplicial volume is multiplicative with respect to finite coverings, while 
for every degree $d$ covering $\widehat{M}\stackrel d\to M$ 
we have 
$$\sigma(\widehat{M})\leqslant d\cdot \sigma(M),$$
which is very often a strict inequality. The \emph{stable $\Delta$-complexity} $\sigma_\infty(M)$ of $M$ is defined by setting
$$\sigma_\infty (M) = \inf_{\widehat{M} \stackrel d\to M} \left\{\frac{\sigma(\widehat{M})}d\right\}$$
where the infimum is taken over all finite coverings $\widehat{M} \stackrel d\to M$ of any finite degree $d$. 
The definition of the stable $\Delta$-complexity, which was introduced by Milnor and Thurston in \cite{MiThu}, is made to be multiplicative with respect to finite coverings. The inequality $\| M,\bb M\|\leq\sigma (M)$ 
and the multiplicativity of the simplicial volume with respect to finite coverings imply that
\begin{equation}\label{stable:in}
 \| M,\bb M\| \leq \sigma_\infty (M)
\end{equation}
for every $n$-manifold $M$. 
It has recently been established~\cite{FFM} that this inequality is strict for closed hyperbolic manifolds of dimension $\geq 4$. 

\subsection*{The simplicial volume of handlebodies}
Every Seifert manifold with nonempty boundary has a finite covering which is the product of a surface
with a circle. Such a covering admits in turn nontrivial self-coverings, and has therefore null stable $\Delta$-complexity.
As a consequence, for every 
Seifert manifold with nonempty boundary
 both the stable $\Delta$-complexity and the simplicial volume vanish. In particular, the inequality~\eqref{stable:in}  is an equality. (The same is true 
for closed Seifert manifolds with infinite fundamental group.) Non zero examples
where the simplicial volume equals the stable $\Delta$-complexity are provided by the following result.%We refer the reader to Subsection~\ref{handles:subsec} for the definition of
%$1$-handle addition.

\begin{thm_intro}\label{general:seif}
Let $M$ be a Seifert manifold with nonempty boundary, and let $N$ be obtained by performing
a finite number of $1$-handle additions on $M$. Then
$$
\| N,\bb N\|=\sigma_\infty (N)=\frac{3}{4}\| \partial N\|\ .
$$
\end{thm_intro}

%In particular, the bound provided by Theorem~\ref{main:thm} is sharp.

%\begin{prop}\label{handles:prop}
%Let $N$ be obtained from $M$ by adding a $1$-handle. Then 
%$$
%\sigma (N)\leq \sigma(M)+3,\qquad \sigma_\infty(N)\leq \sigma_\infty (M)+3\ .
%$$
%\end{prop}  

For every $g\in \matN$ let us denote by $H_g$ the orientable handlebody of genus $g$.
We easily have $\| H_0,\bb H_0\|=\| H_1,\bb H_1\|=0$. Since $H_1$ is a Seifert manifold and 
$H_g$ can be obtained by  performing
$g-1$ handle additions on $H_1$, Theorem~\ref{general:seif} implies the following:
% Proposition~\ref{handles:prop}
%ensures that $\sigma_\infty (H_g)\leq 3(g-1)$ for every $g\geq 2$.
%Putting this result together with 
%inequality~\eqref{stable:in} and Theorem~\ref{main:thm} we get the following:

\begin{cor_intro}\label{handle:thm}
For every $g\geq 2$, the equalities
$$
\| H_g,\bb H_g \|=\sigma_\infty (H_g)=\frac{3}{4}\| \partial H_g\| =3(g-1)
$$
hold.
\end{cor_intro}
This improves the bounds
$$
\frac{4}{3}(g-1)\leq \|H_g,\bb H_g\| \leq \sigma_\infty (H_g)\leq 3(g-1)
$$
exhibited by Kuessner~\cite{Kuessner}. Note that the upper bound also follows from the
 computation of the $\Delta$-complexity of the handlebody $\sigma(H_g)=3g-2$ established by Jaco and Rubinstein \cite{JR}.

\subsection*{Aspherical manifolds with $\pi_1$-injective boundary}
Recall that a connected manifold $M$ is \emph{aspherical} if
$\pi_i(M)=0$ for every $i\geq 2$, or, equivalently, if the universal
covering of $M$ is contractible. If $M$ is disconnected, we say
that $M$ is aspherical if every connected component of $M$ is.
Moreover, we say that $M$ is \emph{boundary irreducible}
%the boundary $\partial M$ is \emph{$\pi_1$-injective}
%in $M$ is 
if for every connected component $B$ of $\partial M$
the inclusion $B\hookrightarrow M$ induces an injective map
on fundamental groups (we borrow this terminology from the context of
$3$-manifold topology, and use it also in the higher dimensional case).
 
The estimate
provided by Theorem~\ref{main:thm} may be improved in the case
of boundary irreducible aspherical 3-manifolds. More precisely,
we prove the following:

\begin{thm_intro}\label{aspherical3:thm}
 Let $M$ be a boundary irreducible aspherical $3$-manifold.
% and assume that
%$\partial M$ is also aspherical. 
Then there is a sharp inequality
$$
\| M,\partial M\|\geq  \frac{5}{4} \|\partial M\|\ .
$$ 
\end{thm_intro}

The equality is realized by products of surfaces with intervals (Corollary~\ref{cor: prod surf}) for which we first compute the 
$\Delta$-complexity. %and Corollary~\ref{asph:3dim} 
Both theorems and their corollary will be proven in Section~\ref{aspherical:sec}.

\begin{thm_intro}\label{prod:surf:thm}
 Let $S_g$ be a closed orientable surface of genus $g\geq 1$ and let
$M_g=S_g\times [0,1]$. Then
$$
 \sigma(M_g)=10(g-1)+6 \ .
$$
\end{thm_intro}

This result has recently
also been obtained by Jaco, Johnson, Spreer and Tillmann using different
methods~\cite{JJST}.
There are remarkably few examples of exact computations of $\Delta$-complexity of manifolds. The first family of examples is given by surfaces,
 Jaco, Rubinstein and Tillmann computed the $\Delta$-complexity of an infinite family of lens spaces \cite{JRT},
and  the $\Delta$-complexity of handlebodies is computed by
Jaco and Rubinstein \cite{JR}. Moreover, a census of closed $3$-manifolds 
up to $\Delta$-complexity 9 and 10 may be deduced from the results in~\cite{M1} and \cite{M2}.
Our Theorem~\ref{prod:surf:thm} provides the exact computation of $\Delta$-complexity for a new infinite family of examples.
 It might be worth mentioning that, in the case of manifolds with boundary, 
the minimal number
of simplices in \emph{ideal} triangulations of manifolds, rather than in (loose) triangulations, has been computed 
for several families of $3$-manifolds. 

\begin{cor_intro}\label{cor: prod surf}
 Let $S_g$ be a closed orientable surface of genus $g\geq 2$ and let
$M_g=S_g\times [0,1]$. Then
$$
\|M_g,\bb M_g\|= \sigma_\infty(M_g)= \frac{5}{4}\|\bb M_g\|\ .
$$
\end{cor_intro}

\section{Simplicial volume}\label{preliminary}
Let $X$ be a topological space and $Y\subseteq X$ a (possibly empty) subspace of $X$. Let
$R$ be a normed ring. In this paper only the cases $R=\R,\Q$ or $\Z$ are considered, where each of these rings is endowed with the norm given by the absolute value.
For $i\in\matN$ we denote by $S_i(X)$ the set of singular $i$-simplices
in $X$, by $C_i(X;R)$ the module
of singular $i$-chains over $R$, and we set as usual $C_i(X,Y;R)=C_i(X;R)/C_i(Y;R)$.
We observe that the $R$-module $C_i (X,Y;R)$ is free and admits the preferred
basis given by the classes of the singular simplices whose image is not contained in $Y$.
Therefore, we will often identify $C_i(X,Y;R)$ with the free $R$-module generated
by $S_i(X)\setminus S_i(Y)$. In particular, for $z\in C_i(X,Y;R)$, it will be understood from the
equality 
$z=\sum_{k=1}^n a_k\sigma_k$ that $\sigma_k\neq\sigma_h$ for $k\neq h$, and $\sigma_k\notin S_i(Y)$ for every $k$. %, while we could also write $z=\sum_{k=1}^n a_k [\sigma_k]$ in $C_i(X;R)/C_i(Y;R)$. 
We denote by $H_\ast (X,Y;R)$ the singular homology of the pair $(X,Y)$ with coefficients
in $R$, \emph{i.e.}~the homology of the complex $(C_{\ast} (X,Y;R),d_\ast)$,
where $d_\ast$ is the usual differential.

We endow the $R$-module $C_i (X,Y;R)$ with the $L^1$-norm defined by
$$
\left\| \sum_{\sigma} a_\sigma\sigma \right\|_R=\sum_{\sigma} |a_\sigma|\ ,
$$
where $\sigma$ ranges over the simplices in 
$S_i (X)\setminus S_i(Y)$. We denote simply by $\|\cdot \|$ the norm $\|\cdot\|_\R$.
%By taking the infimum over representatives, 
The norm $\|\cdot \|_R$
descends to a seminorm on $H_\ast (X,Y;R)$, which is still denoted by $\|\cdot \|_R$ and
is defined as follows:
if $[\alpha]\in H_i (X,Y;R)$, then
$$
\|[\alpha ]\|_R  =  \inf \{\|\beta \|_R,\, \beta\in C_i (X,Y;R),\, d\beta=0,\, [\beta ]=[\alpha ] \} \ .
$$
Note that although $\|\cdot\|_\Z$ is often called a seminorm in the literature, it is technically not so as it is not multiplicative in general. The \emph{real} singular homology module $H_\ast(X,Y;\R)$ 
and the seminorm on $H_\ast(X,Y;\R)$ will be simply denoted
by $H_\ast(X,Y)$ and $\|\cdot\|$ respectively.

%\subsection*{Simplicial volume}
If $M$ is a connected oriented  $n$-manifold with (possibly empty) boundary $\bb M$,
then we denote by $[M,\bb M]_R$ the fundamental class of the pair $(M,\partial M)$ with coefficients
in $R$. 
The following definition is due to Gromov~\cite{Gromov, Thurston}:

\begin{defn}
The \emph{simplicial volume} of $M$ is 
$$\|M,\bb M\|= \|[M,\bb M]_\R\|=\| [M,\bb M]_\R\|_\R\ .$$ 
The \emph{rational}, respectively \emph{integral}, simplicial volume of $M$
is defined as $\|M,\bb M\|_\Q=\|[M,\bb M]_\Q\|_\Q$, resp.~$\|M,\bb M\|_\Z=\|[M,\bb M]_\Z\|_\Z$.
\end{defn}

Just as in the real case, the  rational and the  integral
simplicial volume may be defined also when $M$ is disconnected or nonorientable.
Of course we have the inequalities $\| M,\bb M\|\leq \|M,\bb M\|_\Q\leq \|M,\bb M\|_\Z$.
Using that $\Q$ is dense in $\R$, it may be shown  
in fact that $\|M,\bb M\|=\|M,\bb M\|_\Q$.

The integral simplicial volume 
does not behave as nicely as the rational or real simplicial volume.
For example,
it follows from the definition that $\|M \|_\Z\geq 1$ for every manifold $M$. Therefore, the integral simplicial volume cannot be multiplicative with respect to finite coverings (otherwise
it should vanish on manifolds that admit finite nontrivial self-coverings,
as $S^1$). %Another defect is that the $L^1$-seminorm on integral homology is not really a seminorm, since the equality
%$\|n\cdot \alpha\|_\Z=|n|\cdot \|\alpha\|_\Z$, for $n\in \Z$, $\alpha\in H_\ast(X,Y;\Z)$, may not hold. Indeed, it is easy to see that $\| n\cdot [S^1]\|_\Z=1$ for every $n\in\Z\setminus\{0\}$.
Nevertheless, we will use integral cycles extensively,
as they admit a clear geometric interpretation in terms of pseudomanifolds (see~Section~\ref{main:sec}). 
In order to follow this strategy, we need the following obvious
consequence of the equality $\|M,\partial M\|_\Q=\|M,\partial M\|$.

\begin{lemma}\label{integral:estimate:lemma}
Let $M$ be connected and oriented, and let $\varepsilon>0$ be given. Then, there exists an integral cycle
$z\in C_n(M,\partial M;\Z)$ such that 
$$
\frac{\|z\|_\Z}{d}\leq \|M,\bb M\|+\vare\ ,
$$
where $[z]=d\cdot [M,\bb M]_\Z$ and $d> 0$ is an integer.
\end{lemma}
%\begin{proof}
%Since $\| M,\bb M\|=\| M,\bb M\|_\Q$, a rational %cycle $z'\in\ C_n(M,\partial M;\Q)$
%exists such that $[z']_\Q=[M,\bb M]_\Q$ and $\| %z'\|_\Q\leq \|M,\bb M\|+\vare$.
%Of course there exists $d\in\matN\setminus\{0\}$ %such that $z=d\cdot z'$ 
%lies in $ C_n(M,\partial M;\Z)$. The integral %cycle $z$ satisfies the desired properties.
%\end{proof}

Moreover, the boundary of a fundamental cycle for $M$ is equal to the sum of one fundamental cycle for each component
of $\partial M$, so we also have:

\begin{lemma}\label{easyboundary:lemma}
 Let $M$ be connected and oriented, and let $z$ be an integral $n$-dimensional cycle 
such that $[z]=d\cdot [M,\bb M]_\Z$, where $d> 0$ is an integer. Then
$$
\|\partial M\|\leq \frac{\|\partial z\|_\Z}{d}\ .
$$ 
\end{lemma}
%\begin{proof}
% The chain $z/d$ is a real (in fact, rational)
%fundamental cycle for $M$, so the class
%$[(\partial z)/d]\in H_{n-1}(\partial M;\R)$ 
%is the sum of the real fundamental classes of the %components of $\partial M$,
%and $\|\partial M\|\leq \| (\partial %z)/d\|=\|\partial z\|_\Z/d$.
%\end{proof}
 
 \begin{comment}
\begin{rem} The statements and the proofs of Proposition \ref{rational:prop} and Lemma \ref{integral:estimate:lemma} hold more generally after replacing the fundamental class $[M,\partial M]_{\mathbb{Q}}$ by any rational homology class.
In other words, for every $i\in\mathbb{N}$ the change of coefficients map $H_i(M,\bb M;\mathbb{Q})\to H_i(M,\bb M;\mathbb{R})$
is norm-preserving.
\end{rem}
\end{comment}
 
%\subsection{Elementary properties of the simplicial volume}
Finally, let us list some elementary properties of the simplicial volume which will be needed later. 

\begin{prop}[\cite{Gromov}]\label{degree:prop}
 Let $M,N$ be connected oriented manifolds of the same dimension, and suppose that
either $M,N$ are both closed, or they both have nonempty boundary. Let $f\colon N\to M$
be a map of degree $d$. Then
$$
\|N,\bb N\|\geq |d|\cdot \| M,\bb M\| \ .
$$ 
%If $f$ is a covering, then the equality $\|N\|= d\cdot \| M\|$ holds.
\end{prop}

The following well-known result describes the simplicial volume of 
closed surfaces. In fact, the same statement also holds for connected surfaces with boundary.
 
\begin{prop}[\cite{Gromov}]\label{vol:surf:prop}
Let $S$ be a closed surface. Then
$$
\| S\|=\max \{0,-2\chi (S)\}\ .
$$
\end{prop}

Let $S,S'$ be (possibly disconnected) orientable surfaces without boundary.
We say that $S'$ is obtained from $S$ by an \emph{elementary tubing} if $S'$ is obtained from
$S$ by removing two disjoint embedded disks and gluing an annulus to the resulting boundary components in such a way that the resulting surface is orientable. 
We say that $S'$ is obtained by tubing from $S$ if it is 
obtained from $S$ via 
a finite
sequence of elementary tubings. An immediate application of
Proposition~\ref{vol:surf:prop} yields the following:

\begin{cor}\label{tubing:cor}
Let $S,S'$ be (possibly disconnected) orientable surfaces without boundary, and suppose that $S'$ is obtained
from $S$ by tubing. Then
$$
\| S'\| \geq \|S \|\ .
$$
\end{cor}
Note that for elementary tubing, the inequality is strict unless the tube is attached to a sphere.

\section{Proof of Theorems~\ref{main:thm} and~\ref{general:seif}}\label{main:sec}

\subsection*{Pseudomanifolds and integral cycles}
Let $n\in\matN$. An \emph{$n$-dimensional pseudomanifold} $P$ consists of a finite number of copies
of the standard $n$-simplex, a choice of pairs of $(n-1)$-dimensional faces of $n$-simplices such that each face
appears in at most one of these pairs, and an affine
identification between the faces of each pair. We allow pairs of distinct faces in the same $n$-simplex. 
It is \emph{orientable} if orientations on the simplices of $P$
may be chosen in such a way that the affine identifications between the paired faces
(endowed with the induced orientations)
are all orientation-reversing.
A face which does not belong to any 
pair of identified faces is a \emph{boundary} face.

We denote 
by $|P|$ the \emph{topological realization} of $P$, \emph{i.e.}~the quotient space of the union of the simplices
by the equivalence relation generated by the identification maps.
We say that
$P$ is connected if $|P|$ is.
We denote by $\partial |P|$
the image in $|P|$ of the boundary faces of $P$, and 
we say that $P$ is \emph{without boundary} if $\partial |P|=\emptyset$.

A \emph{$k$-dimensional face}
of $|P|$ is the image in $|P|$ of a $k$-dimensional face of a simplex of $P$.
Usually, we refer to $1$-dimensional, resp.~$0$-dimensional
faces of $P$ and  $|P|$ as to \emph{edges}, resp.~\emph{vertices}
of $P$ and $|P|$.

%Observe that  we do not require 
%the topological realization of a pseudomanifold to be connected.
%In this way, the boundary of a pseudomanifold is itself a pseudomanifold (see below). 

It is well-known that, 
if $P$ is an $n$-dimensional pseudomanifold, $n\geq 3$, then $|P|$ does not need to be a manifold.
However, the topological realization of any $2$-pseudomanifold is a genuine surface, and
singularities of orientable $3$-dimensional pseudomanifolds can occur only at vertices
(see \emph{e.g.}~\cite[pages 108-109]{Hat} and ~\cite[Exercise 1.3.2(b)]{Thurston2}).
Moreover, the boundary of the topological realization of an $n$-dimensional pseudomanifold $P$ is naturally the topological realization
of an (orientable) $(n-1)$-pseudomanifold $\partial P$ without boundary, and an orientation of $P$ canonically induces an
orientation on $\partial P$. In particular, if $P$ is orientable and $3$-dimensional, then $\partial |P|$ is a finite union of orientable
closed surfaces.

Let $M$ be an oriented connected $n$-dimensional manifold with (possibly empty) boundary $\partial M$.
It is well-known that every integral relative cycle on $(M,\partial M)$ can be represented by a map from a suitable pseudomanifold
to $M$. Roughly speaking, let $z=\sum_{i=1}^k \vare_i \sigma_i$ be an $n$-dimensional relative cycle in
$C_n(M,\bb M;\mathbb{Z})$ where
$\vare_i=\pm 1$ for every $i$ (note that here we do not assume that $\sigma_i\neq \sigma_j$ for  $i\neq j$). %In order to describe the associated pseudomanifold %in order to define a pseudomanifold $P$ 
Then one takes
$k$ copies $\Delta_1^n,\ldots,\Delta_k^n$ of the standard $n$-simplex and assume that $\sigma_i$ is defined on $\Delta_i^n$.
Two faces of the $\Delta^n_i$'s are glued to each other if and only if the restriction to them
of the
corresponding $\sigma_i$'s cancel each other in the expression of $\partial z$ as the sum of singular $(n-1)$-simplices. 
The continuous maps $\sigma_i$ glue up to a well-defined continuous map $f\colon (|P|,\bb |P|)\to (M,\bb M)$.
Moreover, a suitable algebraic sum of the simplices of 
$P$ defines a relative cycle $z_P\in C_n(|P|,\partial |P|;\mathbb{Z})$ such that the map $f_*$ induced by $f$ on integral singular chains sends $z_P$ in $f_\ast(z_P)=z\in C_n(M,\partial M;\mathbb{Z})$.
We refer the reader to~\cite[pages 108-109]{Hat} for the details.

\subsection*{Approximating real cycles via pseudomanifolds}
As before, let $M$ be a connected oriented $n$-manifold. 
By Lemma~\ref{integral:estimate:lemma}, the simplicial volume 
of $M$ can be computed from \emph{integral} cycles.
The following proposition shows that such cycles
may be represented by $n$-pseudomanifolds with additional properties.

If $P$ is an $n$-dimensional pseudomanifold, we denote by $c(P)$ the number of $n$-simplices of $P$. If $P$ is associated to the integral cycle $z$, then $c(P)=\| z\|_\Z$ and
$c(\partial P)=\|\partial z\|_\Z$.
%(the 
%obvious inequality $\|\partial z\|_\Z\leq c(\partial P)$ 
%is an equality since by definition the pairings defining $P$ correspond to a \emph{maximal}
%set of canceling pairs of faces).

\begin{prop}\label{cycle:prop}
Let $\vare>0$ be fixed, and suppose that $\|M,\partial M\|>0$. Then, there exists a relative integral $n$-cycle $z\in C_n (M,\partial M;\Z)$ 
with associated pseudomanifold $P$ such that the following conditions hold:
\begin{enumerate}
\item
$[z]=d\cdot [M,\bb M]_\Z$ in $H_n(M,\partial M;\Z)$ for some integer $d> 0$, and
$$ \frac{\|z\|_\Z}{d}\leq \| M,\bb M\| +\vare\ ;$$
\item
$P$ is connected; 
\item
every $n$-simplex of $P$ has at most $n-1$ boundary faces, and in particular
$$\|\partial z\|_\Z\leq (n-1)\|z\|_\Z\ .$$
\end{enumerate}
\end{prop}
\begin{proof}
By Lemma~\ref{integral:estimate:lemma}, we may choose
an integral cycle satisfying condition~(1).

Let us suppose that the pseudomanifold $P$ associated to $z$ is disconnected. 
Then $P$ decomposes into a finite collection of connected pseudomanifolds
$P_1,\ldots,P_k$ such that $c(P_1)+\ldots+c(P_k)=c(P)$. Each $P_i$ represents an integral cycle
$z_i$, and if $d_i\in\Z$ is defined by the equation $[z_i]=d_i[M,\bb M]_\Z$, then $d_1+\ldots+d_k=d$.
Therefore, there exists $i_0\in \{1,\ldots, k\}$ such that $d_{i_0}\neq 0$ and $c(P_{i_0})/|d_{i_0}|\leq c(P)/d$. After
replacing
$z$ with $\pm z_{i_0}$, we may suppose that our cycle $z$ satisfies conditions~(1) and~(2).

As for condition~(3), first note that since $P$ is now connected, if it has an $n$-simplex with $n+1$ boundary faces, 
then $P$ consists of a single simplex $\Delta^n$ with no identifications between its faces. In particular, $z$  
consists of a single singular simplex $\sigma:\Delta^n\rightarrow M$ such that $\sigma(\partial \Delta^n)\subseteq \partial M$,
so $\sigma$ has a well-defined degree, which cannot be null because
$z$ is homologically non-trivial. In other words, there exists a nonzero degree map from the topological ball
to $M$, and this implies that $\|M,\partial M\|=0$, against our assumptions.
\begin{comment}
Let $z_0$ be the chain on $\Delta^n= |P|$ 
obtained by coning $\partial \Delta^n$ to the barycenter $b=1/(n+1)\cdot(1,…,1)\in\Delta^n$ of $\Delta^n$.
 More precisely, $z_0=\sum_{i=0}^n(-1)^i [b,e_0,…,\widehat{e_i},…,e_n]$, where $[b,e_0,…,\widehat{e_i},…,e_n]:\Delta^n\rightarrow \Delta^n$ 
denotes the affine simplex with vertices $b,e_0,…,\widehat{e_i},…,e_n$. Clearly, its integral norm is equal 
to $\| z_0 \|_\mathbb{Z}=\|\partial z_0 \|_\mathbb{Z}=n+1$. Let $f_0:(\Delta^n,\bb \Delta^n)\rightarrow (\Delta^n,\bb \Delta^n)$
 be a map of degree $\geq n+1$ and consider the integral cycle $z'=(\sigma\circ f_0)_*(z_0)$.
 Note that the corresponding pseudomanifold $P'$ is the coning of $\partial \Delta^n$. 
In particular, each of its $n+1$ simplices has exactly one boundary face and $\|\partial z'\|_\Z =n+1 =\|z'\|_\Z\ $ which validates condition~(3). 
Observe also that $P'$ is clearly connected. Furthermore, as $z$ represented $ d\cdot [M,\bb M]_\Z$, the relative cycle $z'$ represents $d\cdot \mathrm{deg}(f_0)\cdot   [M,\bb M]_\Z$ and
$$ \frac{\|z'\|_\Z}{d\cdot \mathrm{deg}(f_0)}=  \frac{n+1}{d\cdot \mathrm{deg}(f_0)}\leq \frac{1}{d}=\frac{\|z\|_\Z}{d}\leq \| M,\bb M\| +\vare\ .$$
\end{comment}
It remains to see that $P$ can be chosen not to have any simplex with $n$ boundary faces. Suppose that $P$ contains a simplex $\Delta_1^n$ having exactly one \emph{non}-boundary face. 
Then $\Delta^n_1$ is adjacent to another simplex $\Delta^n_2$ of $P$. Roughly speaking, we remove
from $z$ the singular simplex $\sigma_1$ corresponding to $\Delta^n_1$, and modify the singular simplex
corresponding to $\Delta^n_2$ by suitably ``expanding'' it
to compensate the removal of $\sigma_1$.
More precisely, if $z=\sum_{i=1}^k \vare_i \sigma_i$, then there exist a map $f\colon (|P|,\partial |P|)\to (M,\bb M)$
and singular simplices $\hat{\sigma_i}\colon \Delta^n\to |P|$ such that
$\sigma_i=f\circ \hat{\sigma}_i$ for every $i=1,\ldots,k$. 
Let us denote by $F^1\subseteq \Delta_1^n$, $F^2\subseteq \Delta_2^n$ the pair of faces glued in $|P|$,
and by $\varphi\colon F^1\to F^2$ their affine  identification. 
Let $Q$ be the space obtained by
gluing $\Delta_1^n$ and $\Delta^n_2$ along $\varphi$. We observe that the
inclusion $\Delta^n_1\sqcup\Delta^n_2\hookrightarrow \Delta^n_1\sqcup\ldots \sqcup\Delta^n_k$
induces a well-defined  
map $\theta\colon Q\to |P|$,
and we fix a homeomorphism $\psi\colon \Delta^n\to Q$ 
that restricts to the identity on $\partial \Delta_2^n\setminus F^2$. Finally, we define the singular simplex
$\sigma_2'\colon \Delta^n\to M$ by setting $\sigma_2'=f\circ \theta\circ\psi$. 

Let us set $z'=\vare_2\sigma_2'+\sum_{i=3}^k \vare_i\sigma_i$, and let $P'$ be the pseudomanifold obtained by removing
from $P$ the simplex $\Delta^n_1$ (and ignoring the only pairing involving a face of $\Delta^n_1$).
It is readily seen that $z'$ is still a relative cycle in $C_n(M,\bb M;\Z)$. What is more, we have
$[z']=[z]=d\cdot [M,\bb M]_\Z$ in $H_n(M,\bb M;\Z)$.  
By construction, $z'$ admits $P'$ as associated pseudomanifold, $|P'|$ is connected and $c(P')=c(P)-1$. As a consequence,
the cycle $z'$ satisfies (1) and (2), and $c(P')<c(P)$. If $z'$ still has some simplex with exactly $n$ boundary faces
of codimension one, then we may iterate
our procedure until we get a cycle satisfying (1) and (2), and having no simplices with $n$ boundary faces. Since at every step the number of simplices of the associated pseudomanifold decreases, this iteration
must come to an end.
\end{proof}

\begin{comment}
The following result is an easy consequence of 
Proposition~\ref{cycle:prop}: 

\begin{prop}\label{boundary2:prop}
If $M$ is an $n$-manifold, where $n\geq 2$, then
$$
\| M,\bb M\|\geq \frac{\|\partial  M\|}{n-1}\  .
$$ 
\end{prop}
\begin{proof}
We may assume that $M$ is connected and oriented. Let $\vare>0$, and choose an integral cycle $z$ as described
in Proposition~\ref{cycle:prop}. 
Then by Lemma~\ref{easyboundary:lemma}  
we get
$$
\|\partial M\|\leq \frac{\|\partial z\|_\Z}{d}\leq \frac{(n-1)\|z\|_\Z}{d}\leq (n-1)(\|M,\bb M\|+\vare)\ ,
$$
which proves the proposition since $\vare$ is arbitrary. 
\end{proof}
\end{comment}

\subsection*{Proof of Theorem~\ref{main:thm}}
%\begin{proof}[Proof of Theorem~\ref{main:thm}]
Let $M$ be a $3$-manifold. We want to prove that
$$
\| M,\bb M\| \geq \frac{3}{4} \| \partial M\|\ .
$$
We may assume that $M$ is connected and oriented. If $\| \partial M\| =0$, there is nothing to prove, so we may assume that $\| \partial M\| >0$,
whence also $\|M,\bb M\|>0$.

Let $\vare>0$ be given. We
choose an integral cycle
$z\in C_3(M,\partial M;\Z)$ with associated
pseudomanifold $P$ that satisfies all the properties described in Proposition~\ref{cycle:prop}.
Recall that $z_P\in C_3(|P|,\partial |P|;\Z)$ is the relative
cycle represented by the (signed) sum of the simplices
of $P$, and that $P$ comes with a map
$f\colon (|P|,\partial |P|)\to (M,\partial M)$ such that $f_\ast(z_P)=z$.

For $i=0,\ldots,4$, let us denote by
$t_i$ the number of $3$-simplices of $P$ having exactly $i$ boundary $2$-faces.
Our choice for $z$ implies that
$t_3=t_4=0$, so
Lemma~\ref{easyboundary:lemma} gives
\begin{equation}\label{estimate2}
d\cdot\|\partial M\|\leq  \|\partial z\|_\Z=c(\partial P)=t_1+2t_2\ .
\end{equation}
Let $S_1,\ldots,S_h$ be the boundary components of $M$, and for every 
$i=1,\ldots,h$, let $Y_i^1,\ldots,Y_i^{k_i}$ be the connected components
of $\partial |P|$ that are taken into $S_i$ by $f$.
Since $P$ is orientable, each $Y_i^j$ is   a closed orientable surface. Let $d_i^j$ be the degree of the map $f|_{Y_i^j}\colon Y_i^j\to S_i$.
Since $[f_\ast(\partial z_P)]=[\partial f_\ast(z_P)]=[\partial z]$
we have $\sum_{j=1}^{k_i} d_i^j =d$ for every $i=1,\ldots,h$, whence %(see Proposition~\ref{degree:prop})
\begin{equation}\label{estimate3}
 \sum_{j=1}^{k_i} \| Y_i^j\|\geq d\cdot \|S_i\|\ 
\end{equation}
(see Proposition~\ref{degree:prop}).

Let us consider the space $H$ obtained by removing from $|P|$ a closed tubular neighbourhood
of edges and vertices. Of course, $H$ is an orientable handlebody, so in particular
$\partial H$ is an orientable
surface. If we denote by $Z_1,\ldots,Z_r$ the boundaries of regular neighbourhoods
of the vertices of $|P|\setminus \partial |P|$, then each $Z_l$ is an orientable surface, and
$\partial H$ is obtained from the union of the $Y_i^j$'s and the $Z_l$'s by tubing. 
Putting together Corollary~\ref{tubing:cor} and inequality~\eqref{estimate3}, this implies
that
\begin{equation}\label{estimate4}
 \| \partial H\|\geq 	\sum_{i=1}^h\sum_{j=1}^{k_i} \| Y_i^j\| +\sum_{l=1}^r \|Z_l\|\geq
\sum_{i=1}^h d\cdot \|S_i\| =d\cdot \|\partial M\|\ .
\end{equation}

Let us denote by $\Gamma$ the graph dual to $P$. Since vertices and edges of $\Gamma$ correspond 
respectively to $3$-simplices and pairs of identified faces of $P$, we compute
\begin{equation*}%\label{chi:eq}
\chi(\Gamma)=(t_0+t_1+t_2)-\frac{4t_0+3t_1+2t_2}{2}=\frac{-2t_0-t_1}{2}\ . 
\end{equation*}
The handlebody $H$ retracts onto $\Gamma$. As a consequence, if $g$
is the genus of $H$ we obtain
\begin{equation}\label{genus:eq}
 1-g=\chi(H)=\chi(\Gamma)=\frac{-2t_0-t_1}{2}\ .
\end{equation}

As we are assuming that $\|\partial M\|>0$, Equation~\eqref{estimate4}
implies that $g\geq 2$, and can be rewritten as 
\begin{equation}\label{chi2:eq}
 4g-4\geq d\cdot \|\partial M\|\ .
\end{equation}
Putting together Equations~\eqref{genus:eq} and \eqref{chi2:eq} we obtain
\begin{equation}\label{chi3:eq}
 4t_0+2t_1\geq d\cdot \|\partial M\|\ .
\end{equation}
Using inequalities~\eqref{estimate2} and \eqref{chi3:eq}
we get
$$
 4(t_0+t_1+t_2)=2(t_1+2t_2)+(4t_0+2t_1)\geq 3d\cdot \|\partial M\|\ ,
$$
whence 
\begin{equation}\label{final:eq}
 \frac{3}{4}\, \|\partial M\|\leq \frac{t_0+t_1+t_2}{d}=\frac{c(P)}{d}=
\frac{\|z\|_\Z}{d}\leq \|M,\bb M\|+\vare \ .
\end{equation}
%where the last inequality is due to the fact that our cycle $z$ was chosen as
%in the statement of Proposition~\ref{cycle:prop}. 
Since $\vare$ is arbitrary,
this concludes the proof of Theorem~\ref{main:thm}.

%\subsection{Adding $1$-handles to $3$-manifolds}\label{handles:subsec}
\subsection*{Proof of Theorem~\ref{general:seif}}
%\begin{proof}[Proof of Theorem~\ref{general:seif}]
If $M$ is a $3$-manifold with nonempty boundary, 
we say that $N$ is obtained from $M$ by adding a $1$-handle if
$N=M\cup_f H$, where $H=D^2\times [0,1]$ is a solid cylinder and $f\colon D^2\times \{0,1\}\to \partial M$ is
a homeomorphism onto the image.

\begin{prop}\label{handles:prop}
Let $M$ be any $3$-manifold with nonempty boundary, and
let $N$ be obtained from $M$ by adding a $1$-handle. Then 
$$
\sigma (N)\leq \sigma(M)+3,\qquad \sigma_\infty(N)\leq \sigma_\infty (M)+3\ .
$$
\end{prop}  
\begin{proof}
Let $T$ be a copy of the standard $2$-simplex, and fix an identification of 
the added $1$-handle  with
the prism
$T\times [0,1]$. For $i=0,1$,
we denote by $\partial_i M$ the boundary component of $\partial M$ glued to $T\times\{i\}$.  Observe
that we may have $\partial_0 M=\partial_1 M$. 
The prism $T\times [0,1]$ can be triangulated by $3$ simplices, in such a way that
$T\times\{0\}$ and $T\times \{1\}$ appear as boundary faces of the triangulation.
Let  $\mathcal{T}$ be a minimal triangulation of $M$.
Every component of $\partial M$ inherits from $\mathcal{T}$ a triangulation with at least two triangles, so we may choose
distinct boundary faces $T_0,T_1$ of $\mathcal{T}$ and affine identifications $\varphi_i\colon T\times \{i\}\to T_i$
in such a way that $N$ is homeomorphic to the space obtained by gluing $M$ and $T\times [0,1]$ along
the $\varphi_i$'s. This space is endowed with a triangulation with
$\sigma(M)+3$ tetrahedra, and this proves that
$$
\sigma(N)\leq\sigma (M)+3\ .
$$

As for the stable $\Delta$-complexity, if $f\colon\widehat{M}\to M$ is any covering
of degree $d$, then it is immediate that $f$ extends to a covering
$\overline{f}\colon \widehat{N}\to N$ of degree $d$, where $\widehat{N}$ is obtained
from $\widehat{M}$ by adding $d$ $1$-handles. We thus get
$$
\frac{\sigma(\widehat{N})}{d}\leq \frac{\sigma(\widehat{M})+3d}{d}=\frac{\sigma(\widehat{M})}{d}+3\ .
$$
Since the covering $f\colon\widehat{M}\to M$ was arbitrary, this implies that
$$
\sigma_\infty(N)\leq\sigma_\infty (M)+3\ ,
$$
and concludes the proof of Proposition~\ref{handles:prop}.
\end{proof}

We can now easily conclude the proof of Theorem~\ref{general:seif}. In fact, 
if $M$ is a Seifert manifold with nonempty boundary then we know from the introduction that $\sigma_\infty (M)=0$. Therefore,
if $N$ is obtained by consecutively adding $h$ handles on $M$, 
Proposition~\ref{handles:prop} implies that 
\begin{equation}\label{final1}
\sigma_\infty (N)\leq \sigma_\infty(M)+3h=3h\ .
\end{equation}
On the other hand, every boundary 
component of $M$ has null Euler characteristic, and no boundary component of any manifold obtained by adding $1$-handles to $M$
has positive Euler characteristic, so Proposition~\ref{vol:surf:prop} implies that $\| \partial N\|=4h$.
Putting together this inequality with Inequality~\eqref{final1}, and recalling that stable $\Delta$-complexity
always bounds the simplicial volume from above, we get
\begin{equation}\label{final2}
\| N,\bb N\|\leq \sigma_\infty (N)\leq \frac{3}{4}\|\partial N\|\ .
\end{equation}
Finally, Theorem~\ref{main:thm} implies that
all the inequalities in~\eqref{final2} are in fact equalities, which finishes the proof of Theorem~\ref{general:seif}. 
%\end{proof}

\section{The case of boundary irreducible aspherical manifolds}\label{aspherical:sec}
The estimate
provided by Theorem~\ref{main:thm} 
may be improved in the case
of boundary irreducible aspherical $3$-manifolds. 
This additional hypothesis can be exploited to construct
a topological straightening for simplices. The existence of such an operator 
implies in turn
that the simplicial volume may be computed just by looking at \emph{straight} chains, that
verify interesting additional topological properties.

\subsection*{A topological straightening}\label{top:straight:sub}
The \emph{straightening procedure} for simplices was introduced by Thurston in~\cite{Thurston}, in order to bound
from below the simplicial volume of hyperbolic manifolds.
If $M$ admits a nonpositively curved Riemannian metric, then
one may associate to every singular simplex $\sigma$ in $M$ a \emph{straight} simplex
$\str(\sigma)$, which
is uniquely determined by the vertices of $\sigma$
and the homotopy classes (relative to the endpoints) of the edges of $\sigma$. 
We now extend this construction to the case we are interested in.

Until the end of the section, we denote by
$M$ a boundary irreducible aspherical manifold such
that $\partial M$ is also aspherical.
Let $p\colon\widetilde{M}\to M$ be the universal covering of $M$, observe
that $\widetilde{M}$ is contractible, and  fix an identification
of $\pi_1(M)$  with the group $\Gamma$ of the covering automorphisms
of $p\colon \widetilde{M}\to M$.
%If $B_1,\ldots,B_h$ are the boundary components of $M$, we
%set $\widetilde{B}_i=p^{-1}(B_i)$. 
Since $M$ is boundary irreducible, the restriction of $p$ to
any connected component of $\partial \widetilde{M}$
%$\widetilde{B}_i$ 
is a universal covering of a component of $\partial M$.
%$B_i$. 
Moreover, since $\partial M$ is aspherical, every component of
$\partial\widetilde{M}$ is contractible. Using these facts and proceeding inductively on the dimension of simplices, one can prove the following:

\begin{prop}\label{straightening:summary}
 For $R=\Z,\Q,\R$, there exists a $\Gamma$-equivariant chain map
 $$
\strtil_*\colon  C_*(\wdtM;R)\to C_*(\wdtM;R)
 $$
 which satisfies the following properties:
\begin{enumerate}
 \item For every simplex $\sigma\in S_i(\wdtM)$, the chain $\strtil_i(\sigma)$ consists of a single simplex. Moreover,
 $\sigma$ and $\strtil_i(\sigma)$ share the same vertices.
 %and the homotopy classes (relative to the endpoints) 
 %of their edges.
\item If two simplices $\sigma,\sigma'\in S_i(\wdtM)$ share the same vertices, 
%and the homotopy classes (relative to the endpoints) 
 %of their edges, 
 then $\strtil_i(\sigma)=\strtil_i(\sigma')$.
 \item If $\sigma\in S_i(\partial \wdtM)$, then $\strtil_i(\sigma)\in S_i(\partial \wdtM)$, so $\strtil_*$ induces a $\Gamma$-equivariant chain map
 $C_*(\wdtM,\bb\wdtM;R)\to C_*(\wdtM,\bb\wdtM;R)$, which will still be denoted by $\strtil_*$.
 \item The map $\strtil_*\colon C_*(\wdtM,\bb\wdtM;R)\to C_*(\wdtM,\bb\wdtM;R)$ is $\Gamma$-equivariantly homotopic to the identity.
 \end{enumerate}
\end{prop}

By the previous proposition, $\strtil_*$ induces a map $\str_*\colon S_*(M)\to S_*(M)$ which induces
in turn a chain map $\str_*\colon C_*(M,\bb M;R)\to C_*(M,\bb M;R)$. 
Simplices that lie in the image of $\str_\ast$
will be called \emph{straight}. 

\begin{lemma}\label{straight:simplex}
 Let $\sigma$ be a straight $k$-simplex, $k\geq 2$, with image in $M$, and suppose that
$\sigma$ is not supported on $\bb M$. Then:
\begin{enumerate}
 \item at most one $(k-1)$-face of $\sigma$ lies on $\bb M$;
\item if $k=3$ and no $2$-face of $\sigma$ lies on $\bb M$, then
at most two edges of $\sigma$ lie on $\bb M$;
\item if $k=3$ then there exist at most three edges of $\sigma$ in $\bb M$. 
\end{enumerate}
\end{lemma}
\begin{proof} 
 Let $\widetilde{\sigma}$ be a fixed lift of $\sigma$ to $\wdtM$. If there exists a component of $\bb\wdtM$ containing
 $m+1$ vertices of $\widetilde{\sigma}$, then there exists a simplex $\widetilde{\sigma}'\in S_k(\wdtM)$ having the same vertices
 of $\widetilde{\sigma}$ and an $m$-dimensional face supported on $\bb\wdtM$. By Proposition~\ref{straightening:summary}, this implies
 that $\widetilde{\sigma}=\strtil_k(\widetilde{\sigma}')$ also has an $m$-dimensional face supported in $\bb \wdtM$.
 In particular, the assumption that $\sigma$ is not supported on $\bb M$ implies that no  component of $\bb\widetilde{M}$ contains all the vertices
of $\widetilde{\sigma}$.

(1) 
 If $\sigma$ has two faces on $\bb M$, then the vertices of $\widetilde{\sigma}$ 
are contained in the same connected component of $\bb \widetilde{M}$, a contradiction.

(2) Suppose  that $\sigma$ has at least three edges on $\bb M$. Since $k=3$, the 
union of the corresponding edges of $\widetilde{\sigma}$ is connected, so
at least three vertices
of $\widetilde{\sigma}$ lie on the same connected component of $\bb\wdtM$, and 
at least one $2$-face of $\sigma$ is 
supported on $\bb M$.

(3) If four edges of $\sigma$ lie on $\bb M$, then as in (2), the union of the corresponding edges of $\widetilde{\sigma}$ is connected. 
But the vertices of these four edges of the $3$-simplex are all the vertices of the $3$-simplex, which all lie on the same connected component of $\bb\wdtM$, 
a contradiction.
\end{proof}

The fact that $M$ supports a relative straightening
can be used to improve Proposition~\ref{cycle:prop} as follows:

\begin{prop}\label{cyclehyp:prop}
Suppose that $M$ is
an aspherical and boundary irreducible $n$-manifold, $n\geq 2$, and that $\bb M$ is also aspherical.
Let $\vare>0$ be fixed and assume that $\|M,\bb M\|>0$. Then, there exists a relative integral cycle $z\in C_n (M,\partial M;\Z)$ 
with associated pseudomanifold $P$ such that the following conditions hold:
\begin{enumerate}
\item
$[z]=d\cdot [M,\bb M]_\Z$ in $H_n(M,\partial M;\Z)$ for some integer $d> 0$, and
$$ \frac{\|z\|_\Z}{d}\leq \| M,\bb M\| +\vare\ ;$$
%\item
%every singular simplex appearing in $z$ is straight;
\item
every simplex of $P$ has at most one $(n-1)$-dimensional boundary face;
\item
if $n=3$, then 
every simplex of 
$P$ without $2$-dimensional boundary  faces has at most
two edges contained in $\bb |P|$ and every simplex has at most three edges in $\bb |P|$.
\end{enumerate}
\end{prop}
\begin{proof}
Let $z'$ be an integral cycle satisfying conditions (1) and (2) of Proposition~\ref{cycle:prop}, and set $z= \str_n(z')\in C_n(M,\bb M;\mathbb{Z})$. 
As usual,
we understand that no simplex appearing in $z$ is supported in $\bb M$ (otherwise,
we may just remove it from $z$ without modifying 
the class of $z$ in $C_n(M,\bb M;\mathbb{Z})$ and decreasing
$\|z\|_\mathbb{Z}$).
Point~(1) descends from the fact that
the straightening operator is norm nonincreasing and homotopic to the identity,
while points~(2) and (3) follow from Lemma~\ref{straight:simplex}. 
\end{proof}

\begin{comment}
Proposition~\ref{cyclehyp:prop} implies the following:

\begin{prop}\label{aspherical:thm}
Let $M$ be a boundary irreducible aspherical $n$-manifold, $n\geq 2$,
and assume that
$\partial M$ is also aspherical. 
Then
$$
\| M,\partial M\|\geq  \|\partial M\|\ .
$$ 
\end{prop}
\begin{proof}
 Take $\vare>0$. If $z$ is chosen as in Proposition~\ref{cyclehyp:prop}, then
$\|\bb z\|_\mathbb{Z}\leq \|z\|_\mathbb{Z}$, so
$$
\|\bb M\|\leq \frac{\|\bb z\|_\mathbb{Z}}{d}\leq \frac{\|z\|_\mathbb{Z}}{d}\leq
\|M,\bb M\|+\vare\, ,
$$
which proves the proposition since $\vare$ is arbitrary.
\end{proof}

\end{comment}

%\section{Proofs of Theorems~\ref{aspherical3:thm} and~\ref{prod:surf:thm} and Corollary~\ref{cor: prod surf}}\label{aspherical3:sec}
\subsection*{Proof of Theorem~\ref{aspherical3:thm}}
%Let us first concentrate our attention on the proof of Theorem~\ref{aspherical3:thm}.
Let us suppose that $M$ is an aspherical boundary irreducible $3$-manifold.
As usual, we may also suppose that $M$ is oriented, and that $\|M,\bb M\|>0$ (otherwise there is nothing to prove).
In order to exploit the machinery introduced above, 
we first reduce to the case
when $\bb M$ is also aspherical. So, let us suppose that a component $S$ of
$\bb M$ is a sphere. Since $M$ is aspherical, $S$ is homotopically trivial, whence homologically trivial, in $M$. This implies that $\bb M=S$, so 
$\|\bb M\|=0$ and the conclusion of Theorem~\ref{aspherical3:thm} is trivially 
satisfied. (In fact, using the Poincar\'e conjecture one can  prove that the
only aspherical $3$-manifold with at least one spherical boundary component is the ball.)
Therefore, henceforth we suppose that $\bb M$ is also aspherical.

We denote by $z$ the cycle provided by Proposition~\ref{cyclehyp:prop}, and by
$P$ the associated pseudomanifold. 
As usual, let $z_P\in C_3(|P|,\partial |P|;\Z)$ be the relative
cycle represented by the (signed) sum of the simplices
of $P$, and let
$f\colon (|P|,\partial |P|)\to (M,\partial M)$ be such that $f_\ast(z_P)=z$.
%If $T$ is a simplex of $P$, then $T$ is endowed with a natural parameterization $\sigma_T\colon \Delta^3\to T\subseteq P$.
%Moreover, there exists $\vare_T\in\{-1,1\}$ such that, if
%$z_P=\sum_T \vare_T\sigma_T$, then $f_\ast(z_P)=z$.
Also recall  that the space $\bb |P|$ is an orientable
surface. 
%If $[\bb P]_\Z$ is the sum of the integral fundamental classes of the components of $\bb |P|$, then
%$[\bb P]_\Z=[\partial z_P]$ (we prefer here the notation $[\bb P]_\Z$ rather than the heavier $[\bb |P|]_\Z$).
%Therefore, 
The equality $f_*(z_P)=d\cdot [M,\bb M]_\Z$ implies that 
%$$
%(f|_{\bb |P|})_*([\bb P]_\Z)=d\cdot [\bb M]_\Z\ ,
%$$
%where $[\bb M]_\Z$ is the sum of the integral fundamental
%classes of the components of $\bb M$. This equality implies that
\begin{equation}\label{firstsimpl}
(f|_{\bb |P|})_*([\bb P])=d\cdot [\bb M]\ ,
\end{equation} 
where $[\bb P]$ (resp.~$[\bb M]$) is the sum of the
\emph{real} fundamental classes of the components of $\bb |P|$ (resp.~of $\bb M$).

Let $\Omega_i$, for $i=0,\ldots,4$, be the set of simplices
of $P$ having exactly $i$ boundary $2$-faces. 
As usual, we denote by $t_i$ the number of elements
of $\Omega_i$. By Proposition~\ref{cyclehyp:prop}
we have $\Omega_2=\Omega_3=\Omega_4=\emptyset$, so
$$
t_2=t_3=t_4=0,\qquad \| z\|_\Z=t_0+t_1\ .
$$
Since $\bb |P|$ admits a triangulation with $t_1$ triangles we have
$t_1\geq \| \bb |P| \|$. Also, Equation~\eqref{firstsimpl}
implies that
$\| \bb |P| \|\geq d\|\bb M\|$, so
\begin{equation}\label{t1:equation}
t_1\geq d\cdot \|\bb M\|\ .
\end{equation}

In order to prove Theorem~\ref{aspherical3:thm} we now need to bound $t_0$ from below.
Let us first introduce some definitions.
We say that an edge $e$ of the $2$-dimensional pseudomanifold $\bb P$ 
is \emph{nice} if $e$ is the edge of a
simplex in $\Omega_0$.
We also say that an edge of $\bb P$ is \emph{bad} if it is not nice.
An edge of the topological realization $\bb |P|$ is \emph{nice} if it is the image of at least one nice edge in $\bb P$.
An edge in $\bb |P|$ is \emph{bad} if it is not nice.
Notice that bad edges in $\bb|P|$ are \emph{not} the image of bad edges in $\bb P$,
since a nice edge in $\bb |P|$ will be the image of a certain number of
nice edges in $\bb P$ and necessarily two bad edges of $\bb P$ corresponding to the
two (possibly nondistinct) tetrahedra having as 2-faces the 2-faces in $\bb |P|$ containing the
original nice edge.

%We say that an edge $e$ of the $2$-dimensional pseudomanifold $\bb P$ 
%is \emph{nice} if $e$ is the edge of at least one
%simplex in $\Omega_0$.
%We also say that an edge of $\bb P$ is \emph{bad} if it is not nice.

\begin{lemma}\label{bad:lemma}
Let $e$ be a bad edge of $\bb |P|$, let
$T_j,T_{j'}$ be the triangles of $\bb P$ adjacent to $e$, and 
let $\Delta^3_j$ (resp.~$\Delta^3_{j'}$) be the simplex of $P$ containing
$T_j$ (resp.~$T_{j'}$). If
$F_j$ (resp.~$F_{j'}$) is the $2$--face of $\Delta^3_j$ 
(resp.~$\Delta^3_{j'}$) such that
$e=F_j\cap T_j=F_{j'}\cap T_{j'}$,  then $F_j,F_{j'}$ are glued to each other in $|P|$.
\end{lemma}
\begin{proof}
%Let us first suppose that $T_j=T_{j'}$ (so in particular $\Delta^3_j=\Delta^3_{j'}$).
%By Lemma~\ref{}, 
%Since every simplex of $P$ has at most one boundary face, we have
%$\Delta^3_j\neq \Delta^3_{j'}$. 
Let $\overline{\Delta^3}$ be the simplex of $P$ glued to $\Delta^3_j$
along $F_j$. Denote by $\overline{F}$ the face
of $\overline{\Delta^3}$ paired to $F_j$. 
We consider separately the following cases.
%We show that each of the following assumptions leads to a contradiction, and this
%will in turn imply the lemma.

(1) 
Suppose that $\overline{\Delta^3}\neq \Delta^3_j$ and 
$\overline{\Delta^3}\neq \Delta^3_{j'}$.
Since $e$ is bad, at least one face $F$ of $\overline{\Delta^3}$ 
has to lie on $\bb |P|$.
Moreover, the conditions
$\overline{\Delta^3}\neq \Delta^3_{j}$, $\overline{\Delta^3}\neq \Delta^3_{j'}$ imply that $F$ cannot
contain $e$, because otherwise $e$ would be adjacent to $3$ boundary 
faces of $P$ (counted with multiplicities). 
Therefore, $\overline{\Delta^3}$ contains four edges which lie on $\bb |P|$, and this contradicts
point (4) of Proposition~\ref{cyclehyp:prop}.

(2) Suppose that $\overline{\Delta^3}=\Delta^3_j$. Since $\Delta^3_j$ has at most
three edges on $\bb |P|$, the unique boundary edge of $F_j$ is paired
to the unique boundary edge of $\overline F$. As a consequence we have
$\overline F=F_{j'}$, as desired (note that in this case we have 
$T_j=T_{j'}$ and $\Delta^3_j=\Delta^3_{j'}$).

(3) Suppose that $\overline{\Delta^3}=\Delta^3_{j'}$. 
Using that $\Delta^3_{j'}$ has at most three edges on $\bb |P|$,  
one may easily show that $\overline F=F_{j'}$, whence the conclusion.
\end{proof}

We denote by $\Gamma\subseteq \bb |P|$ the union of all the nice edges
of $\bb P$. Then, $\Gamma$ is a (possibly disconnected) graph in $\bb |P|$. 

Let $\Gamma_i$, for $i=1,\ldots, s$ be the connected components of $\Gamma$.
For each $i$ we denote by $N_i=N(\Gamma_i)$ a closed regular neighbourhood
of $\Gamma_i$ in $\bb |P|$, chosen in such a way that $N_i\cap T$
is a regular neighbourhood of $\Gamma_i\cap T$ for every triangle
$T$ of $\bb |P|$ and $N_i\cap N_j=\emptyset$ whenever $i\neq j$. We also set $N=\cup_{i=1}^s N_i$ and
$W=\overline{\bb |P|\setminus N}$. Finally, we denote by
$W_1,\ldots, W_r$ the components of $W$ (see Figure~\ref{surfaces1:fig}).

\begin{figure}
\begin{center}
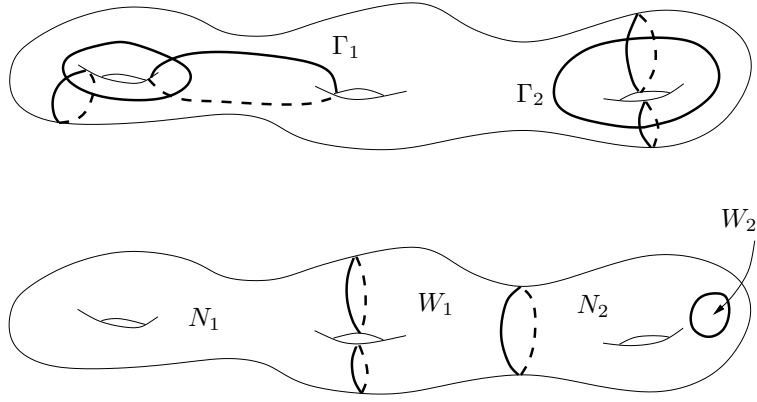
\caption{$N_i$ is a regular neighbourhood of $\Gamma_i$ on $\bb |P|$ for $i=1,2$.\protect\label{surfaces1:fig}}
\end{center}
\end{figure}

We will prove that $f|_{\bb |P|}$ is homotopic to a map which is constant
on each $W_i$. Since $f|_{\bb |P|}$ has degree $d$, this will imply that $N$ has to be sufficiently complicated, and this will prove in turn that the number of nice edges
of $\bb |P|$ cannot be too small.

\begin{lemma}\label{pi1inj:lemma}
Let $\gamma\subseteq W$ be a loop. Then
$f(\gamma)$ is null-homotopic in $\bb M$.
\end{lemma}
\begin{proof}
Let $W_i$ be the component of $W$ containing
$\gamma$. 
Let $\{T_j\}_{j\in J}$ be the family of triangles of $\bb P$ that intersect $W_i$, 
and 
let $\Delta^3_j$ be the simplex in $P$ containing $T_j$. We denote by $v_j$ 
the vertex of $\Delta^3_j$ opposite to $T_j$, and set $\widehat{T}_j=T_j\cap W=T_j\cap W_i$.
Observe that by definition we have $W_i=\cup_{j\in J}\widehat{T}_j$.
Since each edge that intersects $W_i$ is bad, 
Lemma~\ref{bad:lemma} implies that all the vertices $\{v_j\}_{j\in J}$ coincide in $|P|$. This
implies in turn that the inclusion ${W_i}\hookrightarrow\bb |P|$ extends to 
an inclusion $C(W_i)\hookrightarrow |P|$
of the topological cone over $W_i$.
Since $C(W_i)$ is contractible, it follows that $\gamma$ is null-homotopic in $|P|$.
%Every edge that intersects $W_i$ is bad, so Lemma~\ref{bad:lemma}
%implies
%that the inclusion $ W_i\hookrightarrow \bb |P|$ extends to an inclusion
%$C(W_i)\hookrightarrow |P|$, where $C(W_i)$ is the topological cone over $W_i$. Since $C(W_i)$ is contractible, this implies
%in turn that the loop $\gamma$ is null-homotopic in $P$. 
Therefore, $f(\gamma)$ is null-homotopic in $M$, whence in $\partial M$
%Let $i\colon \bb |P|\to |P|$ and $j\colon \bb M\to M$ be the inclusions.
%Since $i(\gamma)$ is null-homotopic in $|P|$, the loop
%$j(f(\gamma))=f(i(\gamma))$ is null-homotopic in $M$. 
due to the boundary irreducibility
of $M$.
%this implies in turn that $f(\gamma)$ is null-homotopic in $\bb M$. 
\end{proof}

\begin{cor}\label{homotopic:cor}
 There exists a map $g\colon \bb |P|\to \bb M$ homotopic to $f|_{\bb |P|}$ and such that
$g|_{W_i}$ is constant for every $i=1,\ldots,r$. 
\end{cor}
\begin{proof}
Each component $W_i$ of $W$ is a compact
orientable surface. Since $\bb M$ is aspherical, by Lemma~\ref{pi1inj:lemma} the map $f|_{W_i}$ can be homotoped to a constant
map $g_i\colon W_i\to \bb M$ via a homotopy $H_i\colon W_i\times [0,1]\to \bb M$
such that $H_i(x,0)=g_i(x)$ and $H_i(x,1)=f(x)$ for every $x\in W_i$. 
We now need to define a global map $g\colon \bb |P|\to \bb M$
such that $g|_{W_i}=g_i$ for every $i$.  

We piecewise define $g$ as follows.
For every component $\gamma$ of $\bb W$ we denote by $N_\gamma$ the component of $N$
containing $\gamma$.  We also fix a collar $C(\gamma)\cong \gamma\times [0,1]$ of $\gamma$ in $N_\gamma$, in such a way that $\gamma$ is identified with
$\gamma\times\{0\}\subseteq C(\gamma)$ and 
all the chosen collars are disjoint, and we set
$$
N'=\overline{N\setminus \bigcup_{\gamma\subseteq \bb W} C(\gamma)}\ .
$$ 
Of course, $N'$ is homeomorphic to $N$. More precisely, there exists a homeomorphism
$t\colon N\to N'$ such that the composition $N\to N'\hookrightarrow N$ is homotopic to the identity of $N$,
and $t(x,0)=(x,1)$ for every $x\in \gamma\subseteq C(\gamma)\cong \gamma\times [0,1]$, where $\gamma$ is any component
of $\bb N=\bb W$.
We  set $g|_{N'}=f|_{N}\circ t^{-1}$. It remains to 
properly define $g$ on the annuli $C(\gamma)$. To this aim, if $\gamma$ is any component of $\bb W_i$
and $(x,s)\in \gamma\times [0,1]\cong C(\gamma)$, then we set
$g (x,s)=H_i(x,s)$. It is easy to check that the resulting map $g$ is well-defined, continuous and homotopic to $f$.
\end{proof}

The following proposition provides the key step in the proof of Theorem~\ref{aspherical3:thm}.  We denote by $E_{\mathrm{nice}}$ the number of nice edges
of $\bb |P|$.

\begin{figure}
\begin{center}
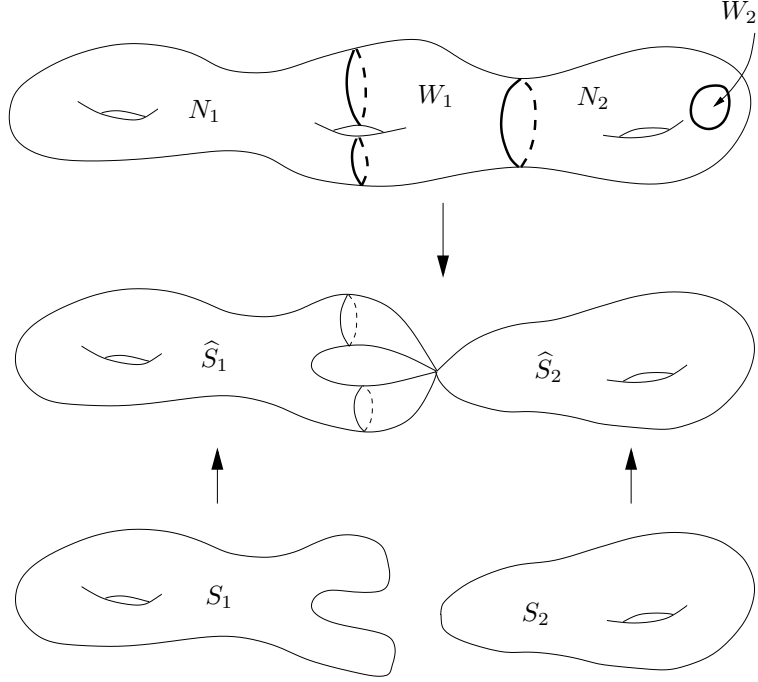
\caption{The construction described in Proposition~\ref{irred:prop}.\protect\label{surfaces:fig}}
\end{center}
\end{figure}

\begin{prop}\label{irred:prop}
We have
$$
2\cdot E_\mathrm{nice}\geq   4\cdot \sharp\{ \mathrm{connected \ components \ of \ } \bb M \} + d\cdot \|\bb M\|   \ .
$$
\end{prop}
\begin{proof}
For every $i=1,\ldots, s$, let $S_i$ be the closed orientable
surface obtained from $N_i$ by collapsing to a point each connected component
of $\bb N_i$ (we understand that distinct components of $\bb N_i$ give rise to distinct points). If $e_i$ is the number
of nice edges of $\Gamma_i$, then
$$
\chi(S_i)\geq 1+\chi(N_i)=1+\chi(\Gamma_i)\geq 2-e_i\ .
$$
Summing over $i$, we obtain
$$
E_\mathrm{nice}\geq 2\cdot s -\sum_{i=1}^s \chi(S_i)\ .$$

Since for closed oriented surfaces $S$ the simplicial volume is equal to $\| S\| =-2\chi(S)$ unless $S$ is homeomorphic to the $2$-sphere in which case $\| S^2 \|=2-\chi(S^2)$, the latter inequality can be rewritten as
$$
2\cdot E_\mathrm{nice}\geq 4\cdot \sharp\{ i \mid S_i\ncong S^2 \} +\sum_{i=1}^s \|S_i\| \ .$$
It remains to show that 
$$  \sharp\{ i \mid S_i\ncong S^2 \} \geq  \sharp\{ \mathrm{connected \ components \ of \ } \bb M\} $$
and
$$ \sum_{i=1}^s \|S_i\| \geq d\cdot \|\bb M\|  \ . $$

Let $\widehat{S}$ be the space obtained from $\bb |P|$ by collapsing to a point each connected component
of $W$ (again, we understand that distinct components of $W$ give rise to distinct points), and let
us denote by $\pi\colon \bb |P|\to \widehat{S}$ the quotient map (see Figure~\ref{surfaces:fig}).
%We also set $\widehat{S}_i=\pi(N_i)\subseteq \widehat{S}$. 
We observe that
$S_i$ canonically projects onto $\widehat{S}$ and we denote by $p:\cup_{i=1}^s S_i\rightarrow \widehat{S}$ the resulting map.

Let us consider the map $g\colon \bb |P|\to \bb M$ provided by Corollary~\ref{homotopic:cor}. 
Being constant on the components of $W$, the map $g$ induces a map $\widehat{g}$ on $\widehat{S}$ and by precomposition by $p$ a map $\widehat{g}_S$ on $\cup_{i=1}^s S_i$ such that the following diagram commutes:
$$
\xymatrix{
\bb |P| \ar[rd]^g \ar[d]^\pi & \\
\widehat{S} \ar[r]_{\widehat{g}} & \bb M \\
\cup_{i=1}^s S_i \ar[ur]_{\widehat{g}_S}\ar[u]_p 
}
$$

The orientation on $\bb |P|$ induces an orientation on each
$N_i$, whence on each $S_i$, endowing $\cup_{i=1}^s S_i$ and also $\widehat{S}= p(\cup_{i=1}^s S_i)$ with a fundamental class. Since $g$ is homotopic to $f|_{\bb |P|}$,
 it follows from the commutativity of the previous diagram and the fact that $\pi_*([\bb |P| ])=[\widehat{S}]=p_*([\cup_{i=1}^s S_i])$ that
\begin{equation}\label{volS3}
 (\widehat{g}_S)_*([\cup_{i=1}^s S_i] )= \widehat{g}_*([\widehat{S}])=g_*([ \bb |P| ])= d\cdot [\bb M] \ , 
 \end{equation}
where the last equality is a consequence of Equation~(\ref{firstsimpl}). 
% We denote by $[\widehat{S}_i]\in H_2(\widehat{S}_i)$ the image of the real fundamental
%class of $S_i$ under the map induced on homology by the projection $S_i\to\widehat{S}_i$, and we set
%$$
%[\widehat{S}]=\sum_{i=1}^s (j_i)_*\left(\left[\widehat{S}_i\right]\right)\ \in\ H_2(\widehat{S})\ ,
%$$
%where $j_i\colon \widehat{S}_i\to \widehat{S}$ is the inclusion. Since continuous maps induce norm non-increasing
%maps in homology we have
%\begin{equation}\label{volS2}
%\| [\widehat{S}] \|\leq \sum_{i=1}^s \|S_i\|\ .
%\end{equation}
Thus, for every connected component $M_0$ of $\bb M$ there exists at least one $S_i$ not homeomorphic to $S^2$ and mapped to $M_0$ by $\widehat{g}_S$, proving the first desired inequality. 
Finally, since  $(\widehat{g}_S)_*$ is norm nonincreasing, we obtain from~\eqref{volS3} the second desired inequality
$$ d\cdot \|\bb M\| = \| (\widehat{g}_S)_*([\cup_{i=1}^s S_i] ) \| \leq  \sum_{i=1}^s \|S_i\| \ , $$
which finishes the proof of the proposition.\end{proof}

To conclude the proof of Theorem~\ref{aspherical3:thm} note that by definition every nice edge of $\bb |P|$ is contained in at least one simplex in $\Omega_0$. Moreover, by point (4) of Proposition~\ref{cyclehyp:prop}
every simplex in $\Omega_0$
has at most
two edges on $\bb |P|$, so the inequality $t_0\geq E_\mathrm{nice}/2$ holds. Putting this inequality together with Proposition~\ref{irred:prop}
and Inequality~\eqref{t1:equation} we get
$$
d (\|M,\bb M\|+\vare)\geq   \| z\|=t_0+t_1\geq d\cdot \frac{\|\bb M\|}{4}+ d\cdot \|\bb M\|=\frac{5d}{4}\| \bb M\|\ ,
$$
which proves the theorem since $\vare$ is arbitrary. 
%\qed

%\begin{prop}\label{stable:above:prop}
%For every $g\geq 2$ we have
%$$
%\sigma(M_g)\leq 10(g-1)+6 \ .
%$$
%\end{prop}
%\begin{proof}[Proof of Theorem~\ref{prod:surf:thm}]%\label{prod:surf:sub}
\subsection*{Proof of  Theorem~\ref{prod:surf:thm}}
We show that the $\Delta$-complexity of the product $M_g=S_g\times [0,1]$ of a surface of genus $g\geq 1$ with an interval is equal to
$$ \sigma(M_g)=10\cdot(g-1)+6\ . $$

For the inequality $\sigma(M_g)\leq 10\cdot(g-1)+6$ we exhibit a topological triangulation of $M_g$ with the prescribed amount of top dimensional simplices. 
Let us realize $S_g$ as the space obtained by gluing the sides of a $4g$-gon, as described in Figure~\ref{octagon:fig}--left, corresponding to the presentation 
$$ \langle x_1,\dots, x_{2g}\mid x_1\cdot \ldots \cdot x_{2g}\cdot (x_1)^{-1}\cdot \ldots \cdot (x_{2g})^{-1})=1\rangle$$
of $\pi_1(\Sigma_g)$ (it is easy to check the cellular structure induced on the quotient has exactly one vertex, so it must be homeomorphic to
$\Sigma_g$ by an obvious Euler characteristic argument). 

We decompose the polygon in $2g-2$ squares and $2$ triangles as indicated in Figure~\ref{octagon:fig}--right. 
We orient each of the edges of that decomposition so that the orientations match on edges being identified in $\Sigma_g$. 
In the product $\Sigma_g \times [0,1]$ a square will be triangulated in $5$ simplices so that for every oriented edge with vertices ${x,y}$, 
the product with $[0,1]$ is triangulated in two simplices by an edge from $\{(x,0)\}$ to $\{(y,1)\}$ if the orientation of the edge goes from $x$ to $y$
(see Figure~\ref{newcube:fig}).
 As for the two products of triangles with $[0,1]$, they can each be triangulated in $3$ simplices respecting the imposed triangulation of their boundaries
(see Figure~\ref{prysm:fig}). 
We have thus constructed a triangulation of $ M_g=S_g\times [0,1]$ in $(2g-2)\cdot 5 + 2\cdot 3$ tetrahedra.

%We denote by $p_i$ and $e_i$, $i=0,\ldots,4g-1$, respectively the vertices and the sides of the $4g$-gon, in such a way that
%$p_{i}$ and $p_{i+1}$ are the vertices of $e_i$.
%We subdivide the $4g$-gon into $4g-2$ triangles $T_1,\ldots,T_{4g-2}$ with $p_0$ as a common vertex, so that
%$e_{i}$ is the side of $T_i$ opposite to $p_0$
%for $i=1,\ldots,4g-2$. 

\begin{figure}[h!]
\begin{center}
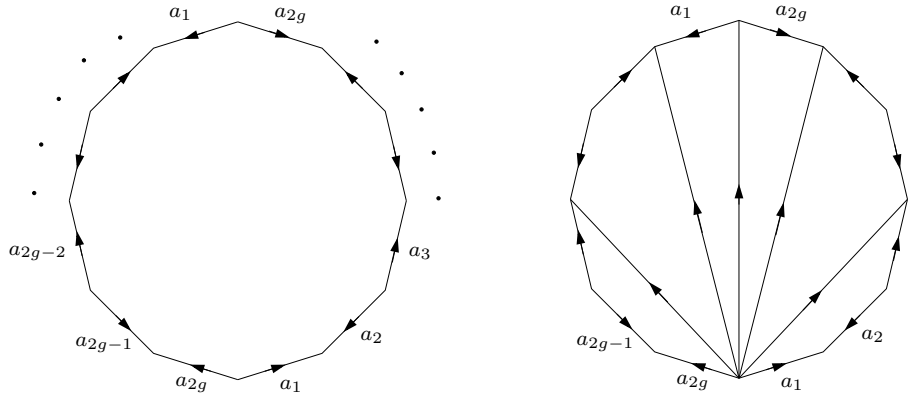
\caption{The decomposition of $S_g$ into $2g-2$ squares and $2$ triangles, in the case $g=3$.\protect\label{octagon:fig}}
\end{center}
\end{figure}

%By taking the product with $[0,1]$,
%the triangulation of $S_g$ just described defines a decomposition of $M_g$ into $4g-2$ triangular prysms $Q_1,\ldots,Q_{4g-2}$
%such that $Q_i=T_i\times [0,1]$ for every $i=1,\ldots,4g-2$. 

%For $i=1,\ldots,4g-2$, let $V_i\subseteq Q_i$ be the tetrahedron 
%with vertices $(p_0,0)$, $(p_0,1)$, $(p_{i},1)$, $(p_{i+1},1)$ 
%and let $W_i=\overline{Q_i\setminus V_i}$ (see Figure~\ref{deco:fig}). Then, each $W_i$
%is a pyramid with rectangular basis $F_i=e_i\times [0,1]$ and vertex $(p_0,0)$.
%We are going to suitably decompose the union of the $W_i$'s in order to get a triangulation of $M_g$
%with $10g-4$ tetrahedra. 

\begin{figure}[h!]
\begin{center}
\includegraphics[width=12cm]{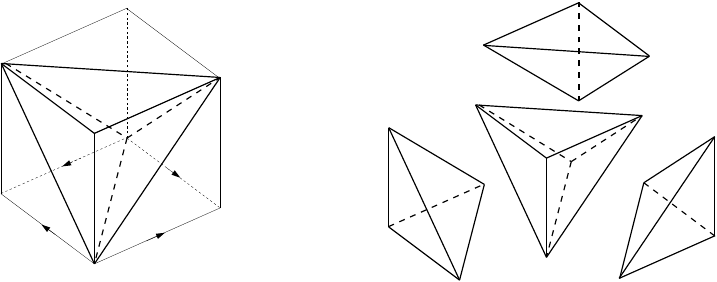}
\caption{Any orientation of the edges of a square $S$ with one source and one sink (on the left) determines a decomposition
of the cube $S\times [0,1]$ into 5 simplices (on the right).}
\protect{\label{newcube:fig}}
\end{center}
\end{figure}

%Let us first observe that the rectangles $F_0$ and $F_2$ are identified in $M_g$.
%We are going to keep the $V_i$'s in our final triangulation of $M_g$, and this 
%choice already defines a subdivision of $F_0$ into two triangles. This forces us to
%subdivide $W_2$ into two tetrahedra $W'_2,W''_2$
%in such a way that the induced triangulation of $F_2$ matches the triangulation of $F_0$.
%In the same way we subdivide  $W_{4g-3}$ into two tetrahedra $W_{4g-3}',W_{4g-3}''$ so to match
%the corresponding subdivision of $F_{4g-1}$.

%Let us consider the remaining pyramids $W_i$, $i\notin \{0,2,4g-3,4g-1\}$.
%These pyramids are glued in pairs along their rectangular faces. More precisely,
%if $e_i$ is paired to $e_j$ in the triangulation of $S_g$ we started with, where $i,j\notin \{0,2,4g-3,4g-1\}$, then
%$W_i$ is glued to $W_j$ along the identified faces $F_i=F_j$. We denote by $Z_{ij}$
%the octahedron that results from such a gluing. 
%We have thus decomposed $M_g$ into the union of $4g+2$ tetrahedra $V_1,\ldots,V_{4g-2}$, $W_2$, $W_2'$,
%$W_{4g-3}'$, $W_{4g-3}''$ and $2g-2$ octahedra. Each of these octahedra
%may be triangulated
 %by $3$ simplices as described in Figure~\ref{octa:tria:fig}, and this shows that $M_g$ may indeed
 %be triangulated with $10g-4$ tetrahedra.
 
 \begin{figure}
\begin{center}
\includegraphics[width=11cm]{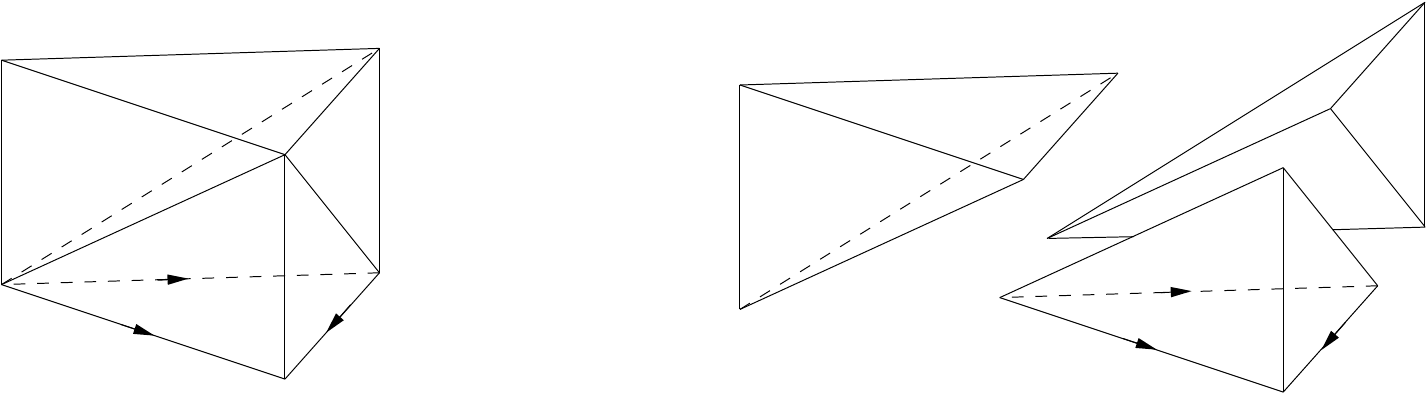}
\caption{A triangle $T$ with oriented edges determines a decomposition of the prysm $T\times [0,1]$ into
3 tetrahedra. \protect{\label{prysm:fig}}}
\end{center}
\end{figure}
 
It remains to prove the other inequality $\sigma(M_g)\geq 10\cdot(g-1)+6$.
Let $T$ be a triangulation of $M_g$ with simplices $\Delta^3_1,\ldots,\Delta^3_N$. 
We need to show that $N\geq 10(g-1)+6$.

We start by choosing a straightening operator on $(M_g,\partial M_g)$ with additional symmetry. To do so, we endow $\Sigma_g$ with a hyperbolic structure 
and consider on $\Sigma_g$ the barycentric straightening, 
associating to any singular simplex $\sigma\colon\Delta^q\rightarrow \Sigma_g$, for $q\geq 0$, 
a straightened simplex $\str_\mathrm{bar}(\sigma)$ (see \cite[Chapter 11]{Ratcliffe} for 
details). Note that the barycentric straightening 
has the property that it does not depend, in constant curvature, 
on the order of the vertices of $\sigma$ (in contrast to the straightening obtained by geodesic coning). 
Consider the resulting product straightening on $\Sigma_g\times [0,1]$, denoted by $\str$, 
where on the interval $[0,1]$ we consider the affine straightening. 
Note that $\str$ still has the property that it does not depend on the order of the vertices. 

For every $i=1,\ldots, N$, we fix an orientation-preserving parameterization
$\sigma_i\colon \Delta^3\to \Delta^3_i$, and we
say that a simplex of $T$ is \emph{inessential} if the image
of $\str(\sigma_i)$ lies on $\bb M_g$, and \emph{essential} otherwise. 
We order the simplices of $T$ so that $\Delta^3_1,\ldots,\Delta^3_{N_0}$ are
essential, and $\Delta^3_{N_0+1},\ldots, \Delta^3_N$ are not.
We now define an orientable pseudomanifold $P$ as follows. The simplices of $P$
bijectively correspond to the essential
simplices of $T$, and gluings in $P$ correspond to gluings between essential
simplices of $T$. (This does not mean that $P$ is identified with a subset of $M_g$, 
since, for example, two 
essential simplices of $T$ may share an edge because they are glued
to the same inessential simplex, so they may intersect in $M_g$, while being disjoint in $P$.)

We define a map $\str_T\colon |P| \to M_g$ 
which corresponds to the simultaneous straightening of all the essential simplices of $T$. To define the map $\str_T$ on $|P|$, we choose, for every $p\in |P|$, an $i\in\{ 1,\ldots,N_0\}$ 
such that $p\in |\Delta^3_i|\subseteq |P|$, and we set $\str_T ( p )=\str(\sigma_i)(q)$,
where $q\in \Delta^3_i\cong\Delta^3$ is any point which gets identified with $p$
under the natural map $\Delta^3_i\to |\Delta^3_i|\subseteq |P|$.
Of course, if the point $p$ 
belongs to the $2$-skeleton of $|P|$, then 
there may be several choices of $i$ and possibly also for $q\in\Delta^3$ (recall that distinct faces of $\Delta^3_i$ may be identified 
in $|P|$). However, since our straightening does not depend on the order of the vertices, one may easily check 
that $\str_T$ is indeed well-defined and continuous. Let us further see that $\str_T$ is a map of pairs
$$
\str_T\colon (|P|,\bb |P|)\to (M_g,\bb M_g)\ .
$$
If $F$ is any boundary face of $|P|$, then either $F$ corresponds to a boundary face of $T$, or 
it corresponds to a face of $T$ which is glued to an inessential simplex of $T$.
In the first case we deduce that $\str_T(F)\subseteq \bb M_g$ from the fact that straightening
preserves the space of singular simplices supported on $\bb M_g$. In the second case 
it is sufficient to observe that after straightening
faces of inessential simplices are supported on $\bb M_g$.

Let us analyze the action of $\str_T$ on fundamental cycles.
We first point out that, in general, we cannot assume that the sum $\sum_{i=1}^N \sigma_i$
is a relative cycle in $C_n(M_g,\bb M_g;\mathbb{Z})$. In fact, if two $2$-faces of $T$
are identified in $M_g$, then 
it may happen that the corresponding faces of the $\sigma_i$'s, when considered as singular
$2$-simplices, differ by the precomposition with a
nontrivial affine
automorphism of the standard $2$-simplex, and do not cancel each other in 
the algebraic boundary of $\sum_{i=1}^N \sigma_i$.
This problem can be fixed by alternating each $\sigma_i$ as follows. 
For any singular $3$-simplex $\sigma$ we define the chain 
$$
\alt(\sigma)=\frac{1}{(4)!}\sum_{\tau\in \mathfrak{S}_{4}} (-1)^{{\rm sgn}(\tau)}\sigma\circ \overline{\tau},
$$ 
where $\overline{\tau}$ is the unique affine diffeomorphism of the
standard $3$-simplex $\Delta^3$ corresponding to the permutation $\tau$ of the vertices of $\Delta^3$. 
Now it is immediate that the real chain 
$z_M^\mathbb{R}=\alt(\sigma_1)+\ldots+\alt(\sigma_N)$ is a cycle which represents the relative
fundamental class of $M_g$.
(Note in passing that this construction may be exploited to prove
the inequality $\| M,\bb M\|\leqslant \sigma (M)$ stated in the introduction, which holds for every 
$3$-manifold $M$). 
Since we know that the straightening operator induces the identity on homology, 
the cycle $\str(z_M^\mathbb{R})$ also 
represents the relative
fundamental class of $M_g$.

The cycle $\str(z_M^\mathbb{R})$ can be realized as the push-forward of a relative
cycle in $C_n(|P|,\bb |P|;\mathbb{R})$ via $\str_T$. To see this, 
we denote by $\hat{\sigma}_i\colon \Delta^3\to |P|$ the singular simplex
corresponding to $\sigma_i\colon \Delta^3\to M_g$, $i=1,\ldots,N_0$. Here also, the sum
of the $\hat{\sigma}_i$'s does not provide in general an integral relative cycle for $(|P|,\bb |P|)$,
so we have to recur to the real relative cycle 
$z_P^\mathbb{R}=\alt(\hat{\sigma}_1)+\ldots+\alt (\hat{\sigma}_{N_0})$. 
By contruction, the chains $(\str_T)_*(z_P^\mathbb{R})$ and 
$\str(z_M^\mathbb{R})$ differ just by a linear combination of simplices supported on $\bb M_g$.
As a consequence, they define the same element of $C_n(M_g,\bb M_g;\mathbb{R})$, so
$(\str_T)_*(z_P^\mathbb{R})$ is a relative fundamental cycle of $M_g$. 

Since $P$ is an orientable $3$-dimensional pseudomanifold, every connected component
of $\bb |P|$ is a closed orientable surface. If we denote by $[\bb P]$ the sum of the real
fundamental classes of the components of $\bb |P|$, then our previous considerations imply that
$(\str_T)_*([\bb P])$ is equal to the sum of the fundamental classes
of the boundary components $\bb_0 M_g$, $\bb_1 M_g$ of $M_g$. In particular, for $i=0,1$,
there exist components $\bb_i |P|$ of $\bb |P|$ such that the restriction
$\str_T|_{\bb_i |P|}\colon \bb_i |P|\to \bb_i M_g$ has positive degree. This implies
that %$\|\bb_i P\|\geq \|\bb_i M_g\|$, so 
$g_i\geq g$, where $g_i$ is 
the genus of $\bb_i |P|$. 
%(recall that $g\geq 1$).

As usual, we denote by $t_j$ the number of simplices of $P$ with exactly $j$ boundary faces,
and we recall that $t_2=t_3=t_4=0$, so that $N_0=t_0+t_1$.
Since the $\Delta$-complexity of a closed surface of genus $\overline g$ is $4\overline{g}-2$ we get
\begin{equation}\label{t1:c}
t_1\geq (4g_0-2)+(4g_1-2)\geq 8g-4\ .
\end{equation}

Just as in the computations leading to Theorem~\ref{aspherical3:thm}, 
we now need to bound $t_0$ from below. To this aim we exploit Proposition \ref{irred:prop} with $d=1$. 
Actually, the proposition is stated for pseudomanifolds associated to integral cycles, but the proof carries through without changes in our present setting of a pseudomanifold $P$ with a map of pairs
$$
\str_T\colon (|P|,\bb |P|)\to (M_g,\bb M_g)\ 
$$
sending the real fundamental class of $(|P|,\bb |P|)$ to the real fundamental class of $(M_g,\bb M_g)$. Keeping notations and terminology from above,
we denote by $E_{\mathrm{nice}}$ the number of nice edges of $\bb |P|$,
and we recall that 
$t_0\geq E_{\mathrm{nice}}/2$. 
Since $M_g$ has two boundary components,
Proposition~\ref{irred:prop} implies that
$$
t_0\geq \frac{E_{\mathrm{nice}}}{2}\geq \frac{8+\| \bb M_g \|}{4}=2g\ .
$$
Putting together this inequality with Inequality~\eqref{t1:c} we get 
$$
N\geq N_0=t_0+t_1\geq 2g+8g-4=10g -4\ ,
$$
which concludes the proof.
%\end{proof}

\subsection*{Proof of Corollary~\ref{cor: prod surf}}
%\begin{proof}[Proof of Corollary~\ref{cor: prod surf}]
Using that the stable $\Delta$-complexity bounds the simplicial volume from above and applying Theorem~\ref{aspherical3:thm}, we have the inequalities
$$
\sigma_\infty(M_g)\geq \|M_g,\bb M_g\|\geq \frac{5}{4}\|\bb M_g\|=10(g-1)\ .
$$
It remains to see that the stable $\Delta$-complexity of $M_g$ is smaller than or equal to $10(g-1)$. For every $d\geq 2$ the manifold $M_g$ admits a covering of degree $d$ whose total space is
homeomorphic to $M_{g'}$, where $g'=d(g-1)+1$. By Theorem~\ref{prod:surf:thm}, this implies
that
$$
\sigma_\infty(M_g)\leq \frac{\sigma (M_{d(g-1)+1})}{d}\leq \frac{10d(g-1)+6}{d}=10(g-1)+\frac{6}{d}\ .
$$
Since $d$ is arbitrary, the corollary is proved.
\bibliographystyle{amsalpha}
\bibliography{biblio}

\end{document}